\documentclass[11]{amsart}

\usepackage{amsfonts}
\usepackage{amscd}
\usepackage{amsmath, mathrsfs, amssymb}
\usepackage{amsthm}
\usepackage{setspace}
\usepackage{hyperref}
\usepackage{color}
\usepackage{epsfig}
\usepackage{here}
\usepackage{graphicx}
\usepackage[all]{xy}
\usepackage{psfrag}
\usepackage{graphicx,transparent}
\usepackage{enumerate}
\usepackage{caption}

\theoremstyle{plain}
\newtheorem{theorem}{Theorem}[section]

\newtheorem{corollary}[theorem]{Corollary}

\theoremstyle{definition}
\newtheorem{remark}[theorem]{Remark}

\newtheorem{example}[theorem]{Example}

\newcommand{\MM}{\mathcal M}

\newcommand{\BM}{\overline{\mathcal M}}

\newcommand{\PP}{\mathcal P}

\newcommand{\CC}{\mathcal C}
\newcommand{\calC}{\mathcal C}

\newcommand{\OO}{\mathcal O}

\newcommand{\RR}{\mathbb R}
\newcommand{\BHH}{\overline{\mathcal H}}
\newcommand{\BPP}{\overline{\mathcal P}}

\newcommand{\HH}{\mathcal H}

\newcommand{\mult}{\operatorname{mult}}
\newcommand{\bbP}{\mathbb P}

\newcommand{\bbZ}{\mathbb Z}

\newcommand{\SL}{\operatorname{SL}}

\newcommand{\Res}{\operatorname{Res}}

\begin{document}

\makeatletter
	\@namedef{subjclassname@2010}{%
	\textup{2010} Mathematics Subject Classification}
	\makeatother

\title{Positivity of divisor classes on the strata of differentials}

\author{Dawei Chen}
\address{Department of Mathematics, Boston College, Chestnut Hill, MA 02467}
\email{dawei.chen@bc.edu}

\subjclass[2010]{14H10, 14H15, 14E30}
\keywords{strata of differentials, moduli space of curves, positivity of divisor classes}

\date{\today}

\thanks{The author is partially supported by NSF CAREER Award DMS-1350396.}

\begin{abstract}
Three decades ago Cornalba-Harris proved a fundamental positivity result for divisor classes associated to families of stable curves. 
In this paper we establish an analogous positivity result for divisor classes associated to families of stable differentials. 
\end{abstract}

\maketitle


\section{introduction}
\label{sec:intro}

For a positive integer $k$, let $\mu = (m_1, \ldots, m_n)$ be an integral (ordered) partition of $k(2g-2)$, i.e., $m_i \in \bbZ$ and $\sum_{i=1}^n m_i = k(2g-2)$. The stratum of $k$-differentials $\HH^k(\mu)$ parameterizes $(C, \xi, p_1, \ldots, p_n)$, where $C$ is a smooth and connected complex algebraic curve of genus $g$ and $\xi$ is a (possibly meromorphic) section of $K_C^{\otimes k}$ such that $(\xi)_0 - (\xi)_\infty = \sum_{i=1}^n m_i p_i$ for distinct (ordered) points $p_1, \ldots, p_n \in C$. If we consider differentials up to scale, then the corresponding stratum of $k$-canonical divisors $\PP^k(\mu)$ parameterizes the underlying divisors $ \sum_{i=1}^n m_i p_i$, and $\PP^k(\mu)$ is thus the projectivization of $\HH^k(\mu)$. We remark that 
for special $\mu$ the strata $\HH^k(\mu)$ and $\PP^k(\mu)$ can be disconnected with extra hyperelliptic, spin or nonprimitive connected components \cite{KontsevichZorich, LanneauQuad, Boissy}. 

Abelian and quadratic differentials, i.e., the cases $k=1$ and $k=2$ respectively, have broad connections to flat geometry, billiard dynamics, and Teichm\"uller theory. Recently their algebraic properties (also for general $k$-differentials) have been investigated significantly, which has produced remarkable results in the study of Teichm\"uller dynamics and moduli of curves. We refer to \cite{ZorichFlat, WrightLectures, Bootcamp, EskinMirzakhani, EskinMirzakhaniMohammadi, Filip, FarkasPandharipande, BCGGM1, BCGGM2} for related topics as well as some recent developments. 

Let $\MM_{g,n}$ be the moduli space of smooth genus $g$ curves with $n$ distinct (ordered) marked points. For a signature $\mu = (m_1, \ldots, m_n)$ with $m_1, \ldots, m_r \geq 0$ and $m_{r+1}, \ldots, m_n < 0$, define the (twisted) $k$th Hodge bundle 
$\HH^k(\mu^{-})$ over $\MM_{g,n}$ (twisted by the negative part $\mu^{-} = (m_{r+1}, \ldots, m_n)$ of $\mu$) as 
$\pi_{*} (\omega^{\otimes k} (-m_{r+1}S_{r+1} - \cdots - m_n S_n))$, where $\pi: \CC \to \MM_{g,n}$ is the universal curve, $\omega$ is the relative dualizing line bundle and $S_i$ is the section of $\pi$ corresponding to the marked point $p_i$. Then $\HH^k(\mu^{-})$ contains $\HH^k(\mu)$ as a subvariety. The (twisted) $k$th Hodge bundle $\HH^k(\mu^{-})$ extends and remains as a vector bundle $\BHH^k(\mu^{-})$ over the Deligne-Mumford compactification $\BM_{g,n}$ of pointed stable curves. Denote by $\BPP^k(\mu^{-})$ the projectivization of $\BHH^k(\mu^{-})$  over $\BM_{g,n}$ and let $\BPP^k(\mu)$ be the closure of $\PP^k(\mu)$ in $\BPP^k(\mu^{-})$. The compactified stratum $\BPP^k(\mu)$ is called the incidence variety compactification of $\PP^k(\mu)$ which parameterizes pointed stable differentials (up to scale) of type $\mu$ \cite{Gendron}. The incidence variety compactification $\BPP^k(\mu)$ carries richer information than the Deligne-Mumford strata compactification. The boundary points of $\BPP^k(\mu)$ were completely characterized in \cite{BCGGM1, BCGGM2}. Denote by $\eta$ the $\OO(-1)$ line bundle class of the projective bundle $\BPP^k(\mu^{-})$ as well as its restriction to each compactified stratum $\BPP^k(\mu)\subset \BPP^k(\mu^{-})$. We remark that $\eta$ plays a significant role in the tautological ring of the strata 
\cite{Tauto}. 

Recall that the rational Picard group of $\BM_{g,n}$ is generated by $\kappa$ (or $\lambda$), $\psi_1, \ldots, \psi_n$ and boundary divisor classes, where $\kappa = \pi_{*}c_1^2(\omega)$ is the first Miller-Morita-Mumford class, $\lambda$ is the first Chern class of the (ordinary) Hodge bundle, and $\psi_i$ is the cotangent line bundle class associated to the $i$th marked point. Denote by $\psi = \psi_1 + \cdots + \psi_n$. We remark that our definition of the $\kappa$ class differs from the one defined in \cite{ArbarelloCornalba}, where $\kappa^{{\rm AC}}  = \kappa + \psi$. We still use $\kappa$, $\lambda$ and $\psi_i$ to denote the pullbacks of these classes to $\BPP^k(\mu)$. We also denote by $\delta$ the total boundary divisor class of $\BM_{g,n}$ and its pullback to $\BPP^k(\mu)$. The relation $\kappa = 12\lambda - \delta$ holds on $\BM_{g,n}$, and hence it also holds on $\BPP^k(\mu)$. 
 
For a compactified moduli space, understanding positivity of divisor classes can provide crucial information for its birational geometry. For the moduli space $\BM_g$ of stable genus $g$ curves, Cornalba-Harris \cite{CornalbaHarris} showed that certain linear combinations of $\kappa$ (or $\lambda$) and $\delta$ are nef. They also obtained a similar result (and independently by Xiao \cite{Xiao}) for families of stable curves not contained in the boundary of $\BM_g$ (see Section~\ref{subsec:known} for a precise statement 
and related results). In this paper we establish an analogous positivity result for divisor classes associated to families of pointed stable differentials. 

\begin{theorem}
\label{thm:nef}
Let $\pi: \calC \to B$ be a family of pointed stable $k$-differentials $(C, \xi)$ (up to scale) of type $\mu = (m_1, \ldots, m_n)$ over an integral and complete curve. Then 
\begin{eqnarray}
\label{eq:arbitrary}
\deg_{\pi} k (\kappa + \psi) - \eta \geq 0, 
\end{eqnarray}
i.e., the divisor class $ k (\kappa + \psi) - \eta$ is nef on $\BPP^k(\mu)$. 

Moreover if $\xi$ is not identically zero on any irreducible component of the generic fiber of $\pi$, then 
\begin{eqnarray}
\label{eq:psi}
 \deg_{\pi} (m_i + k) \psi_i - \eta \geq 0 
 \end{eqnarray}
for all $i$ and
\begin{eqnarray}
\label{eq:generic} 
 \deg_{\pi} k (\kappa + \psi) - (2g-2+ n) \eta \geq 0. 
 \end{eqnarray}
\end{theorem}

\begin{corollary}
\label{cor:ample}
The divisor class $a (\kappa +  \psi)  - b \eta$ is ample on $\BPP^k(\mu)$ if $a > kb > 0$. 
\end{corollary}

We will prove Theorem~\ref{thm:nef} and Corollary~\ref{cor:ample} in Section~\ref{sec:nef}. Moreover in Section~\ref{sec:example} we will construct special families of pointed stable differentials to show that our results are optimal in general (see e.g., Corollary~\ref{cor:iff}). Note that since $\eta$ is the $\OO(-1)$ class of the Hodge bundle, linear combinations of $\eta$ with $\kappa$ and $\psi$ span important sections in the Neron-Severi group of the strata. In particular, the divisor classes involving $\eta$ in Theorem~\ref{thm:nef} are not pullbacks from $\BM_{g,n}$, hence the results cannot be deduced directly by using positivity of divisor classes on $\BM_{g,n}$. 

Our results and methods have at least three potential applications. First, the incidence variety compactification $\BPP^k(\mu)$ is a shadow of a more delicate modular compactification parameterizing twisted differentials which are not identically zero on any irreducible component of the underlying curves \cite{BCGGM3}. Since the construction in \cite{BCGGM3} is complex analytic, one cannot directly conclude about the projectivity of the resulting moduli space. Nevertheless, it is promising that one can apply the idea in this paper to exhibit an ample divisor class associated to families of twisted differentials, hence to prove the projectivity of the moduli space of twisted differentials. In addition, knowing ample divisor classes can help determine the birational type of the underlying variety, e.g., the canonical class of a variety of general type can be written as the sum of an ample divisor class and an effective divisor class. We hope to compare the canonical class of the strata with the ample divisor classes produced in this paper in order to determine the birational type of the strata (see \cite{Gendron, Barros, Affine} for some partial results). Finally, the second part of Theorem~\ref{thm:nef} yields an approach to analyze the base loci of divisor classes outside of the ample cone of the strata, hence it can help us better understand the effective cone decomposition of the strata as well as alternate birational models in the sense of Mori's program. We plan to treat these questions in future work. 

\subsection*{Acknowledgements} The author thanks Maksym Fedorchuk and Quentin Gendron for communications on relevant topics. 

\section{Preliminaries}
\label{sec:prelim}

In this section we collect some preliminary results related to our study. 

\subsection{Positivity of divisor classes on $\BM_{g,n}$}
\label{subsec:known}

We recall some known results about positivity of divisor classes on the moduli space $\BM_{g,n}$ of pointed stable curves (see \cite[Chapter XIV]{ACG} for a detailed survey). 

First consider the case of curves without marked points. Let $\pi: \calC \to B$ be a family of (semi)stable curves of genus $g$ over an integral and complete curve. Suppose that the generic fiber of $\pi$ is smooth. Cornalba-Harris \cite{CornalbaHarris} and Xiao \cite{Xiao} proved independently the following slope inequality  
$$ \left(8 + \frac{4}{g}\right) \deg_{\pi} \lambda - \deg_{\pi} \delta \geq 0. $$
Under the same assumption Moriwaki improved this inequality by attributing different weights to singular fibers of different kinds \cite{Moriwaki}. Moreover, in the same paper Cornalba-Harris showed that $a\lambda - b\delta$ is ample on $\BM_g$ if and only if $a > 11b > 0$. 

In general for $\BM_{g,n}$, $\lambda$ and $\psi_i$ are known to be nef, dating back to Arakelov and Mumford \cite{MumfordStability, MumfordEnumerative}. Moreover, Cornalba showed that $12\lambda - \delta + \psi = \kappa + \psi $ is ample on $\BM_{g,n}$ \cite{Cornalba}. Fedorchuk further generalized such positivity results to the moduli space of weighted pointed stable curves \cite{Fedorchuk}.

\subsection{Structure of the $k$th Hodge bundle}
\label{subsec:hodge}

 We analyze the structure of the (twisted) $k$th Hodge bundle over a special kind of nodal curves. The result will be used in Section~\ref{sec:example} when we construct special families of pointed stable differentials. 
 
 Let $C$ be the union of pointed smooth curves $(C_1, q_1)$ and $(C_2, q_2)$ with $q_1$ and $q_2$ identified as a node $q$, where $C_i$ has genus $g_i$ and $g_1 + g_2 = g$. Denote by $\omega_C$ the dualizing line bundle of $C$. Let $D_i$ be an effective divisor in $C_i$ whose support does not contain $q_i$. Consider the space $H^0(C, \omega_C^{\otimes k} (D_1 + D_2))$ whose elements $\xi = (\xi_1, \xi_2)$ consist of (possibly meromorphic) $k$-differentials $\xi_i$ in $H^0(C_i, K_{C_i}^{\otimes k}(kq_i + D_i))$ such that the $k$-residue condition 
 $\Res^k_{q_1} \xi_1 = (-1)^k \Res^k_{q_2} \xi_2$ holds (see \cite[Proposition 3.1]{BCGGM2} about $k$-residue). Note that if $q_1$ is not a $k$th order pole of $\xi_1$, then $\Res^k_{q_1} \xi_1 = 0$, hence $\Res^k_{q_2} \xi_2 = 0$, which is satisfied if $q_2$ is not a $k$th order pole of $\xi_2$. In particular, 
the subspace $H^0(C_1, K_{C_1}^{\otimes k})\oplus (0) \subset H^0(C, \omega_C^{\otimes k} (D_1 + D_2))$ is independent of the position of the node $q$ and the moduli of $C_2$. We remark that the divisor $D_1 + D_2$ corresponds to twisting by the negative part of $\mu$ in the setting of the $k$th Hodge bundle. 

\section{Positivity of divisor classes on $\BPP^k(\mu)$}
\label{sec:nef}

In this section we prove Theorem~\ref{thm:nef} and Corollary~\ref{cor:ample}. 

\begin{proof}[Proof of Theorem~\ref{thm:nef}]
We first prove \eqref{eq:psi} and \eqref{eq:generic}. Let $\pi: \calC\to B$ be a one-parameter family of pointed stable $k$-differentials in $\BPP^k(\mu)$ whose generic fiber is smooth, with distinct sections $S_1, \ldots, S_n$ as the zeros and poles of the parameterized $k$-differentials. After base change and replacing $\calC$ by its minimal desingularization if necessary, we can assume that $\calC$ is smooth and each fiber is a (semi)stable pointed curve (see e.g., \cite[p. 308]{HarrisMorrison}). Denote by $\omega$ the relative dualizing line bundle class of $\pi$. If a special fiber of the family is reducible, then the $k$-differential it carries can be identically zero along some irreducible components of the fiber. Denote by $V\subset \calC$ the union of these irreducible components (with suitable multiplicities) on which the $k$-differentials are identically zero. Since $\eta$ represents the line bundle $\OO(-1)$ of generating $k$-differentials of the family, comparing $\pi^{*}\eta$ with $\omega^{\otimes k}$, the former has zeros (or poles) along the sections $S_i$ with multiplicity $m_i$ and has zeros along the support of $V$. In other words, the following relation of divisor classes holds on $\calC$: 
\begin{eqnarray}
\label{eq:tauto-1}
 \pi^{*} \eta = k\omega - \sum_{i=1}^n m_i S_i - V. 
\end{eqnarray}
Intersecting both sides of \eqref{eq:tauto-1} with $S_i$ and applying $\pi_*$, we obtain that 
$$ \deg_\pi \eta = (m_i + k) \deg_\pi \psi_i - S_i \cdot V. $$
Since $S_i$ is a section and $V$ consists of components of special fibers, we have $S_i \cdot V \geq 0$. Hence 
$ (m_i + k) \psi_i - \eta $ has nonnegative degree for the family $\pi$, thus proving \eqref{eq:psi}. 

Now rewrite the above as 
$$ (m_i + k)\deg_\pi \eta = (m_i + k)^2 \deg_\pi  \psi_i - (m_i+k) (S_i \cdot V). $$
Summing over $i=1, \ldots, n$ and using $\sum_{i=1}^n m_i = k (2g-2)$, we have 
$$ k(2g-2 + n)\deg_\pi  \eta = \sum_{i=1}^n (m_i + k)^2 \deg_\pi \psi_i - \sum_{i=1}^n (m_i+k) (S_i \cdot V). $$
Since $\pi_{*}V = 0$, intersecting \eqref{eq:tauto-1} with $V$ and applying $\pi_{*}$, we obtain that 
$$ 0 = k (\omega \cdot V) - \sum_{i=1}^n m_i (S_i \cdot V) - V^2. $$
Moreover, taking the square of \eqref{eq:tauto-1} and applying $\pi_{*}$, we obtain that 
$$ 0 = k^2 \deg_\pi \kappa - \sum_{i=1}^n (m_i^2 + 2km_i)\deg_\pi \psi_i + V^2 - 2k (\omega \cdot V) + 2 \sum_{i=1}^n m_i (S_i \cdot V). $$
Using these relations, we conclude that  
\begin{eqnarray*}
 k^2\deg_\pi (\kappa + \psi)- k(2g-2+n)\deg_\pi \eta & = & k^2 \deg_\pi\kappa - \sum_{i=1}^n (m_i^2 + 2km_i)\deg_\pi \psi_i \\
  &\  & +\sum_{i=1}^n(m_i+k)(S_i \cdot V) \\
 & = &  2k(\omega \cdot V) - V^2 - \sum_{i=1}^n (m_i - k)(S_i \cdot V) \\ 
 & = & k(\omega + \sum_{i=1}^n S_i)\cdot V \geq 0,
\end{eqnarray*}
since $\omega + \sum_{i=1}^n S_i$ is nef on $\calC$ (see e.g., \cite{Keel}). Therefore, $k(\kappa+  \psi) - (2g-2+n)\eta $ has nonnegative degree for the family $\pi$, thus proving \eqref{eq:generic}. 

Next we prove \eqref{eq:arbitrary}. Consider an arbitrary one-parameter family $\pi: \calC \to B$ of pointed stable $k$-differentials in $\BPP^k(\mu)$. Suppose that the normalization of $\pi$ consists of irreducible families $\pi_i: \calC_i \to B$ with generic smooth fibers of genus $g_i$. Denote by $\Gamma_\ell$ the sections in the normalization arising from the preimages of the generic nodes in $\calC$. Each $\Gamma_\ell$ is paired uniquely with another $\Gamma_{\ell'}$ to form the corresponding generic node. Denote by $\psi_\ell$ the cotangent line bundle class associated to $\Gamma_\ell \subset \calC_i$, so that $\deg_{\pi_i} \psi_\ell = - \Gamma_\ell^2$. 

If $\deg_{\pi}\eta \leq 0$, then \eqref{eq:arbitrary} holds, since $\kappa +\psi$ is ample on $\BM_{g,n}$. Now we assume that $\deg_{\pi}\eta > 0$. Note that 
$$ \deg_{\pi} \lambda = \sum_i \deg_{\pi_i} \lambda \quad \mbox{and} \quad \deg_{\pi}\delta = \sum_i \deg_{\pi_i}\delta + \sum_{\ell} \Gamma^2_{\ell}. $$
It follows that 
\begin{eqnarray*}
\deg_{\pi}(\kappa +  \psi) & = & \sum_i (12 \deg_{\pi_i} \lambda - \deg_{\pi_i}\delta + \sum_{S_j \subset \calC_i} \deg_{\pi_i}\psi_j - \sum_{\Gamma_{\ell}\subset \calC_i}  \Gamma^2_{\ell}) \\
& = & \sum_i (12 \deg_{\pi_i} \lambda - \deg_{\pi_i}\delta + \sum_{S_j \subset \calC_i} \deg_{\pi_i}\psi_j + \sum_{\Gamma_{\ell}\subset \calC_i} \deg_{\pi_i}\psi_\ell). 
\end{eqnarray*}
By the definition of $\BPP^k(\mu)$, there exists some $i$ such that the parameterized $k$-differential is not identically zero on the generic fiber of $\pi_i$. Then by \eqref{eq:generic} we have 
\begin{eqnarray*}
k(12 \deg_{\pi_i} \lambda - \deg_{\pi_i}\delta + \sum_{S_j \subset \calC_i} \deg_{\pi_i}\psi_j + \sum_{\Gamma_{\ell}\subset \calC_i} \deg_{\pi_i}\psi_{\ell}) & \geq & (2g_i - 2 + n_i) \deg_{\pi} \eta \\
&\geq &  \deg_{\pi}\eta  
\end{eqnarray*}
where $n_i$ is the number of sections $S_j$ and $\Gamma_{\ell}$ lying in $\calC_i$ and the inequality $2g_i - 2 + n_i \geq 1$ follows from stability. Moreover for all $i$, we have 
$$12 \deg_{\pi_i} \lambda - \deg_{\pi_i}\delta + \sum_{S_j \subset \calC_i} \deg_{\pi_i}\psi_j + \sum_{\Gamma_{\ell}\subset \calC_i} \deg_{\pi_i}\psi_{\ell} \geq 0, $$
as $12\lambda - \delta + \sum \psi_j + \sum \psi_\ell $ is ample on the corresponding moduli space of pointed stable curves when we treat both $S_j$ and $\Gamma_\ell$ as sections of marked points. Combining the above analysis, \eqref{eq:arbitrary} thus follows. 
\end{proof}

Now we prove Corollary~\ref{cor:ample}. 

\begin{proof}[Proof of Corollary~\ref{cor:ample}]
It suffices to show that the divisor class 
$D = (k + c) (\kappa + \psi) - \eta$ with $c > 0$ is ample. 
Let $f: \BPP^k(\mu) \to \BM_{g,n}$ be the morphism forgetting the differentials. Since $\kappa + \psi$ is ample on $\BM_{g,n}$, by Seshadri's criterion (see e.g., \cite[(6.32)]{HarrisMorrison}) there exists $\alpha > 0$ such that 
$$ (\kappa + \psi)\cdot C \geq \alpha  \mult_{x} C $$
for every integral curve $C\subset \BM_{g,n}$ and every $x\in C$. Suppose that $C'$ is an integral curve in $\BPP^k(\mu)$ and $x'\in C'$. If $C'$ is not contracted by $f$, then the image $C$ of $C'$ under $f$ remains to be an integral curve. Suppose that $C'$ has degree $d$ over $C$ and let $x = f(x')$. It follows that $ d \mult_{x} C\geq \mult_{x'}C' $. Since 
$k (\kappa + \psi) -  \eta$ is nef on $\BPP^k(\mu)$ by Theorem~\ref{thm:nef}, we have 
$$D\cdot C' \geq c (\kappa + \psi)\cdot (dC) \geq c \alpha (d \mult_{x} C)\geq c\alpha \mult_{x'}C'. $$
If $C'$ is contracted by $f$, then $C'$ is contained in a fiber of the (twisted) $k$th Hodge bundle over $\BM_{g,n}$. Since $\eta$ restricted to each fiber of the Hodge bundle is the $\OO(-1)$ class, it implies that 
$$ D\cdot C' = (-\eta) \cdot C' \geq \mult_{x'} C'. $$ 
Combining the two cases above, by Seshadri's criterion we thus conclude that $D$ is ample on $\BPP^k(\mu)$. 
\end{proof}

\section{Special families of differentials}
\label{sec:example}

We first explain that equality can be attained for \eqref{eq:psi} and \eqref{eq:generic} in Theorem~\ref{thm:nef}. In the proof we see that if 
$V = 0$, then the divisor classes $(m_i + k) \psi_i - \eta$ and $k (\kappa + \psi) - (2g-2+ n) \eta$ have degree zero on $B$. This is the case, for instance, when $B$ corresponds to a Teichm\"uller curve in the strata of abelian or quadratic differentials, as degenerate differentials parameterized in the boundary of Teichm\"uller curves do not possess any identically zero component (see \cite{ChenMoellerAbelian, ChenMoellerQuadratic}).  

Next we construct special families of pointed stable differentials that have nonpositive intersection with divisor classes outside of the ample range in Corollary~\ref{cor:ample}. 

\begin{example}
\label{ex:1}
Take $(C, p_1, \ldots, p_n, q)$ in the stratum $\PP^k(m_1, \ldots, m_n, -2k)$ of genus $g-1$ with $q$ as the last pole such that the $k$-residue at $q$ is zero. According to \cite{GendronTahar}, such $(C, p_1, \ldots, p_n, q)$ exists for $g\geq 2$ and any $k, \mu$ except 
$g = k = 2$ and $\mu = (3,1)$ (but the stratum $\PP^2(3,1)$ is empty). 
Glue $C$ to a pencil $B$ of plane cubics $E$ by identifying $q$ with a base point of the pencil. For the resulting pointed stable curve $C$ union $E$, let $\xi$ be a stable $k$-differential such that $\xi|_C \equiv 0$ and $\xi|_E$ is the unique nontrivial holomorphic $k$-differential on $E$ (up to scale). This yields a family $\pi$ of pointed stable $k$-differentials contained in $\BPP^k(\mu)$ by \cite{BCGGM2}. We have $\deg_{\pi} \psi_i = 0$, $\deg_{\pi} \kappa = 1$ and $\deg_{\pi} \eta = k$,
where the last identity follows from the fact that $\eta = k\lambda$ over $\BM_{1,1}$ since the $k$th Hodge bundle on 
$\BM_{1,1}$ is a line bundle with Chern class $k\lambda$. It implies that for $g\geq 2$ and any $k, \mu$ 
the divisor class 
$a (\kappa + \psi) - b \eta$ is not ample when $a \leq k b$. 
\end{example}

\begin{example}
\label{ex:2}
Take $(C, p_1, \ldots, p_n, q_1, q_2)$ in the stratum $\PP^k(m_1, \ldots, m_n, -2k, -2k) $ of genus $g-2$, 
where $q_1$ and $q_2$ correspond to the last two poles of order $2k$ such that the $k$-residue at $q_i$ is zero. According to \cite{GendronTahar}, such $(C, p_1, \ldots, p_n, q_1, q_2)$ exists for $g \geq 3$ and any $k, \mu$ except $g=3$, $k= 2$ and
$\mu =  (5, 3)$. 
Attach $C$ to two elliptic tails $E_1$ and $E_2$ at $q_1$ and $q_2$ respectively to form a pointed stable curve. Let $\xi_i$ be a nontrivial holomorphic $k$-differential on $E_i$. Let $B \cong\bbP^1$ be the family parameterizing pointed stable differentials  
$\xi_{[t_1, t_2]} = (E_1, t_1 \xi_1; C, 0; E_2, t_2 \xi_2)$, where $[t_1, t_2]$ are the homogeneous coordinates of $\bbP^1$. By \cite{BCGGM2} the family $\pi$ of differentials $\xi_{[t_1, t_2]}$ over $B$
is contained in $\BPP^k(\mu)$. Since the underlying pointed stable curve is fixed, we have $\deg_{\pi} \kappa = \deg_{\pi} \psi_i = 0$. 
Since $B$ corresponds to a line in a fiber of the $k$th Hodge bundle and $\eta$ is the $\OO(-1)$ class, we have $\deg_{\pi} \eta = -1$. 

For the exceptional case $g=3$, $k=2$ and $\mu =  (5, 3)$, take $(\bbP^1, p_1, p_2, q_1, q_2, q_3)$ in the stratum $\PP^2(5,3,-4, -4, -4)$ such that the $2$-residues at the three poles $q_1, q_2, q_3$ are of the form $r, r, 0$ for some $r\neq 0$. By \cite{GendronTahar} such $(\bbP^1, p_1, p_2, q_1, q_2, q_3)$ exists. Glue an elliptic bridge $E_1$ to $\bbP^1$ at $q_1, q_2$ and glue an elliptic tail $E_2$ to $\bbP^1$ at $q_3$. Let $\xi_i$ be a nontrivial holomorphic quadratic differential on $E_i$. Let $B \cong\bbP^1$ be the family parameterizing pointed stable differentials  $\xi_{[t_1, t_2]} = (E_1, t_1 \xi_1; \bbP^1, 0; E_2, t_2 \xi_2)$. By \cite{BCGGM2} the family $\pi$ of differentials $\xi_{[t_1, t_2]}$ over $B$
is contained in $\BPP^2(5,3)$. By the same reason we have $\deg_{\pi} \kappa = \deg_{\pi} \psi_i = 0$ and $\deg_{\pi} \eta = -1$. 

Combining the above cases, it follows that for $g\geq 3$ and any $k, \mu$ the divisor class 
$a (\kappa + \psi) - b \eta$ is not ample when $b \leq 0$. 
\end{example}

\begin{example}
For $g\geq 2$, fix a differential $(C, \xi)$ in the stratum 
$\HH^k(k(2g-4))$ of genus $g-1$. Take a nontrivial one-parameter family $B$ of pointed genus one curves in the stratum $\BPP^k(m_1, \ldots, m_n, -k(2g-2))$ with zero $k$-residue at the last pole. According to \cite{GendronTahar}, such one-parameter family of differentials exists for $g\geq 2$ and any $k, \mu$ except $ g = k = 2$ and $\mu = (3,1)$ (but the stratum $\PP^2(3,1)$ is empty). Identify the section of poles of order $-k(2g-2)$ in the family with the zero of order $k(2g - 4)$ in $C$. This way we obtain a family $\pi: \calC \to B$ with differentials given by $\xi$ on $C$ and identically zero on the genus one components. By \cite{BCGGM2} this family is contained in $\BPP^k(\mu)$. By the analysis in Section~\ref{subsec:hodge}, the differentials parameterized by $B$ are contained in the subspace $H^0(K_C^{\otimes k})\oplus (0)$ which is independent of the genus one components. Since $\xi$ is fixed on $C$, we conclude that 
$\deg_{\pi} \eta = 0$. It implies that for $g\geq 2$ and any $k, \mu$ the divisor class $a (\kappa + \psi) - b \eta$ is not ample when $a \leq 0$. 
\end{example}

Combining these examples with Corollary~\ref{cor:ample}, we thus conclude the following. 

\begin{corollary}
\label{cor:iff}
For $g\geq 3$ and any $k, \mu$, the divisor class $a (\kappa +  \psi)  - b \eta$ is ample on $\BPP^k(\mu)$ if and only if $a > kb > 0$. 
\end{corollary}

\begin{remark}
The ``only if'' statement in Corollary~\ref{cor:iff} does not hold for $g\leq 2$ in general. For example, consider $g=2$, $k=1$ and $\mu = (2)$. One can check that $\BPP^1(2)$ is isomorphic to the closure of the stratum in $\BM_{2,1}$, which is the divisor parameterizing genus two curves with a marked Weierstrass point. Since $\kappa +  \psi$ is ample 
on $\BM_{2,1}$, it is also ample on the subvariety $\BPP^1(2)$, hence in this case $b > 0$ is not necessary for the divisor class $a (\kappa +  \psi)  - b \eta$ to be ample. Similarly consider $g=1$ and $\mu = (0)$, i.e., the stratum of holomorphic $k$-differentials on elliptic curves. Note that $\BPP^k(0)$ is isomorphic to $\BM_{1,1}$, on which 
$\kappa = 0$, $\psi = \lambda$ and $\eta = k\lambda$. It follows that $a (\kappa + \psi) - b \eta = (a-kb) \lambda$ is ample if and only if $a > kb$ (but $a$ and $b$ do not have to be positive). For $g=0$, consider $\mu = (m_1, m_2, m_3)$ with 
$m_1 + m_2 + m_3 = -2k$. Then the stratum $\BPP^k(m_1, m_2, m_3)$ consists of a single point, hence $a (\kappa +  \psi)  - b \eta$ is ample for any $a$ and $b$. 
\end{remark}


\end{document}